\documentclass{amsart}

\newcommand{\N}{\mathbb{N}}
\newcommand{\R}{\mathbb{R}}
\newcommand{\C}{\mathbb{C}}

\newtheorem{theorem}{Theorem}[section]
\newtheorem{lemma}[theorem]{Lemma}
\newtheorem{proposition}[theorem]{Proposition}

\theoremstyle{definition}

\newtheorem{example}[theorem]{Example}

\theoremstyle{remark}
\newtheorem{remark}[theorem]{Remark}

\numberwithin{equation}{section}

\newcommand{\esssup}{\operatorname*{ess\,sup}}

\begin{document}

\title{Stochastic Stability for Fiber Expanding Maps via a Perturbative Spectral Approach}

\author{Yushi Nakano}

\address{Graduate School of Human and Environmental Studies, 
Kyoto University 
Yoshida Nihonmatsu-cho, Sakyo-ku,
Kyoto, 606-8501,
Japan}
\email{
nakano.yushi.88m@st.kyoto-u.ac.jp}
\thanks{This work was supported by JSPS KAKENHI Grant Number 11J01842.
}



\begin{abstract}
We consider small perturbations of expanding maps induced by skew-product mappings whose base dynamics are not invertible necessarily. 
Adopting a previously developed perturbative spectral approach, 
we show stability of the densities of the unique absolutely continuous invariant probability measures for expanding maps under these perturbations,  and upper bounds on the rate of exponential decay of fiber correlations associated to the measures as the noise level goes to zero.
\end{abstract}

\maketitle

\markboth{Y. Nakano}{STOCHASTIC STABILITY VIA A PERTURBATIVE SPECTRAL APPROACH}

\section{Introduction}
As is well known, several statistical properties of dynamical systems (such as the existence of SRB measures, the exponential decay of correlations, and the central limit theorem)~ can be obtained by demonstrating the spectral gap of the transfer operator of the dynamical system in a suitable Banach space.
In addition, these statistical properties and quantities are expected to be stable if "the spectrum of the transfer operator" is also stable. 
(The precise definition and properties of the transfer operator are provided in Section \ref{section:proof}.)
This perturbative spectral approach was developed by Baladi and Young and their contemporaries, who sought a simple proof that a (piecewise) expanding map is stochastically stable (i.e.,  the densities of the unique absolutely continuous invariant probability measures for the dynamics are stable) under independent and identically distributed  perturbations, and that its related statistical quantities, such as the rate of the exponential decay of correlations, are also stable (see  \cite{Baladi} and references therein). 
This approach was 
extended by Baladi \cite{Baladi97} and independently by Bogensch\"utz \cite{B1}, to the case of perturbations induced by skew-product mappings. However, these extensions are restricted to mixing or invertible base dynamics. 
In this paper, an alternative perturbative spectral approach based on the Baladi-Young perturbation lemmas is presented, in which the  base dynamics need not be mixing or invertible. 
Consequently, stochastic stability and upper bounds of the exponential decay of correlations for expanding maps under perturbations induced by skew-product mappings whose base dynamics are not invertible necessarily are demonstrated. 
Our result extends the result established by Baladi, Kondah, and Schmitt in \cite{BKS}. 

\subsection{Definitions and results}\label{results}
Let $C^r(M,M)$ be the space of all $C^r$ endomorphisms on a compact smooth Riemannian manifold $M$, endowed with the usual $C^r$ metric $d_{C ^r}(\cdot ,\cdot)$ with $r> 1$. 
(Given that $r=k+\gamma$ for some $k\in \N$, $k\geq 1$ and $0\leq \gamma \leq 1$, $f\in C^r(M,M)$ denotes the $k$-th derivative of $f$ is $\gamma$-H\"older.)
$f$ in $C^r(M,M)$ is said to be an expanding map when there exist constants $C>0$ and $\lambda >1$ such that
\begin{equation*}
\Vert Df^n(x)v\Vert \geq C\lambda ^n \Vert v\Vert  ,\quad n\geq 1
\end{equation*}
for each $x\in M$ and $v\in T_xM$. For the properties of expanding maps, the reader is referred to \cite{KH}. 
The expanding constant $\Lambda _r(f)$ of an expanding map $f:M\to M$ is defined by
\begin{equation*}
\Lambda _r(f)=\limsup _{m\to \infty}\left(\sup _{x\in M} \sum _{f^m(y)=x} \frac{\Vert D(f_{y}^{-m})(x)\Vert ^r}{\vert \det Df^m(y)\vert } \right)^{1/m},
\end{equation*}
which is strictly smaller than $1$ (see (2.16) in \cite{BKS}).
Here, $f^{-m}_{y}$ is the corresponding local inverse branch in a neighborhood of $x$ for each $y\in f^{-m}(\{x\})$. 

Let $\varOmega$ be a separable metric space endowed with the Borel $\sigma$-field $\mathcal B(\varOmega)$ with complete probability measure $P$.
Given an expanding map $f_0:M\to M$ of class $C^r$,
let $\{ f_{\epsilon} \} _{\epsilon >0}$ be a family of continuous mappings defined on $\varOmega$ with values in $C^r(M,M)$ such that
\begin{equation}\label{eq:assumption}
\esssup\displaylimits_{\omega\in \varOmega} d_{C^r}(f_{\epsilon}(\omega) ,f_0) \rightarrow 0 \quad \mathrm{as} \;\epsilon \rightarrow 0.
\end{equation}
For each $\epsilon >0$, 
adopting the notation $f_\epsilon (\omega ,\cdot)=f_\epsilon (\omega)$, 
the distance between $f_\epsilon(\omega ,x)$ and $f_\epsilon(\omega ^\prime ,x)$ is bounded by $d_{C^r}(f_\epsilon (\omega),f_\epsilon (\omega ^\prime))$ for each $x\in M$ and each $\omega , \omega ^\prime \in \varOmega$. Thus,
it is straightforward to realize that $f_\epsilon :\varOmega \times M\to M$ is a continuous (in particular, measurable) mapping.
 Note also that if $\epsilon >0$ is sufficiently small, $f_\epsilon (\omega)$ is $P$-almost surely an expanding map of class $C^r$. 

Let
$\theta :\varOmega \rightarrow \varOmega$ be a measure-preserving measurable transformation on $(\varOmega ,P)$. 
For each $\epsilon >0$ and $n\geq 1$, let $f_{\epsilon}^{(n)}(\omega,x)$ be the fiber component in the $n$-th iteration of the skew product mapping
\begin{equation*}
\varTheta _{\epsilon} (\omega , x) =(\theta \omega ,f_{\epsilon}(\omega ,x)),\quad (\omega ,x) \in \varOmega\times M,
\end{equation*}
where we simply write $\theta \omega$ for $\theta (\omega)$. 
Setting the notation $f_{\epsilon}^{(n)}(\omega)=f_{\epsilon}^{(n)}(\omega,\cdot)$,the explicit form of $f_{\epsilon}^{(n)}(\omega)$ is
\begin{equation}
\label{eq:fibre}
f_{\epsilon}^{(n)}(\omega )=f_{\epsilon}(\theta^{n-1}\omega)\circ f_{\epsilon}(\theta^{n-2}\omega)\circ\cdots\circ f_{\epsilon}(\omega).
\end{equation}

In \cite{BKS} and other articles on fiber dynamics, $\theta$ is required to be a bimeasurable transformation, i.e., an invertible measurable transformation whose inverse mapping is also measurable (see, for example, \cite{Buzzi1, B1, GQ}; a significant exception is described in Baladi \cite{Baladi97}). However, some framework accommodates important examples that are not generally invertible, as shown in Example \ref{examples}.  
Let $L^p_{\nu}(S)$ be the usual $L^p$ space on a measurable space $(S, \Sigma ,\nu)$ endowed with the $L^p$ norm  
$\Vert \cdot \Vert _{L^p}$ where $1\leq p\leq \infty$.
For each $u\in L^\infty_P(\varOmega)$, a functional $\ell _\theta u:L^1_P(\varOmega) \to \C$ is defined as by
$
\ell _\theta u(\varphi) 
 =\int u(\omega) \cdot \varphi (\theta \omega) dP$ for each  $ \varphi \in L^1_P(\varOmega)$.
Since $P$ is an invariant measure, $\vert \ell _\theta u (\varphi)\vert \leq \Vert u\Vert _{L^\infty} \Vert  \varphi \circ \theta \Vert _{L^1}= \Vert u\Vert _{L^\infty} \Vert  \varphi  \Vert _{L^1}$, i.e., $\Vert \ell _\theta u\Vert _{(L^1_P(\varOmega))^\ast} \leq \Vert u\Vert _{L^{\infty}}$. Thus, by the Riesz representation theorem, $\ell _\theta u\in L^\infty_P(\varOmega)\cong (L^1_P(\varOmega))^\ast$ and $\ell _\theta :L^\infty _P(\varOmega) \to L^\infty_P(\varOmega)$ is a bounded operator on $L^\infty _P(\varOmega)$ such that 
\begin{equation}\label{eq:A1}
 \int \ell _\theta u(\omega) \cdot \varphi (\omega)dP
 =\int u(\omega) \cdot \varphi (\theta \omega) dP, \quad  \varphi \in L^1_P(\varOmega).
\end{equation}
 ($\ell _\theta$ is called the transfer operator of $\theta $ with respect to $P$.)

Let $C^{r-1}(M)$ be the space of all complex-valued functions on $M$ of class $C^{r-1}$ endowed with the usual $C^{r-1}$ norm $\Vert \cdot \Vert _{C^{r-1}}$, and let $m$ be the normalized Lebesgue measure on $M$. 
Let $L^{\infty}_P (\varOmega ,C^{r-1}(M))$ be the Lebesgue-Bochner space of mappings defined on $\varOmega$ taking values in the Banach space $C^{r-1}(M)$ endowed with the $L^\infty$ norm $\Vert u\Vert _{L^{\infty}}:=\esssup _{\omega \in \varOmega}\Vert u (\omega)\Vert _{C^{r-1}}$. Here the 
usual abuse of notation is adopted (where an $L^\infty$ mapping is identified by its equivalence class).
The definition and properties of this space are provided in 
\cite{DJ77}. Here it is merely stated that 
 if $u\in L^\infty_P(\varOmega ,C^{r-1}(M))$, then $u$ is Bochner measurable, i.e., $u=\lim _{n\to \infty}u_n$ $P$-almost surely, where $u_n:\varOmega \to C^{r-1}(M)$ is a simple function of each $n\geq1$. 
 Setting $u(\omega ,\cdot)=u(\omega)$, for each $x\in M$ the mapping $\omega \mapsto u(\omega ,x)$ is $P$-almost surely the limit of the sequence $\{u_n(\cdot ,x)\}_{n\geq 1}$ of simple functions, and is thus measurable because $P$ is a complete probability measure. 
 Furthermore, $\Vert u(\cdot ,x)\Vert _{L^\infty}\leq \Vert u\Vert _{L^\infty}$; that is, $u(\cdot ,x) \in L^\infty_P(\varOmega)$ for each $x\in M$. 
It is supposed that for $\ell _\theta$ (and therefore $\theta$), there exists a bounded operator $\tilde \ell _\theta$ on $L_P^\infty(\varOmega ,C^{r-1}(M))$ such that the following holds for each $u\in L^\infty_P(\varOmega ,C^{r-1}(M))$, each bounded linear functional $A:C^{r-1}(M)\to \C$, each bounded operator $\mathcal A:C^{r-1}(M)\to C^{r-1}(M)$, each $x\in M$ and $P$-almost every $\omega \in \varOmega$:
\begin{equation}\label{eq:A2}
\tilde \ell _\theta u(\omega ,x)=\ell _\theta [u(\cdot ,x)](\omega)
\end{equation}
\begin{equation}\label{eq:A4}
\ell _\theta  [  Au (\cdot)](\omega)= A\tilde \ell _\theta u(\omega),
\quad \tilde \ell _\theta  [ \mathcal Au (\cdot)](\omega)= \mathcal A\tilde \ell _\theta u(\omega),
\end{equation}
and
\begin{equation}\label{eq:A3}
\Vert \tilde \ell _\theta u \Vert _{L^\infty} \leq \Vert u\Vert _{L^\infty} . 
\end{equation}

Now, some definitions are provided on measure-preserving skew-product transformations. 
Let $\mathcal{B} (M)$ be the Borel $\sigma$-field of $M$. 
It is known that for each probability measure $\mu$ on $\varOmega \times M$ with marginal $P$ on $\varOmega$, there exists a function $\mu _{\cdot}(\cdot):\Omega \times \mathcal{B} (M) \to [0,1]$, $(\omega,B)\mapsto \mu_{\omega}(B)$ that satisfies the following three conditions: $\omega \mapsto \mu _{\omega}(B)$ is measurable for each $B\in \mathcal{B} (M)$; $\mu _{\omega}$ is $P$-almost surely a probability measure on $M$; $\int \varphi d\mu =\int \varphi d\mu _{\omega} dP$ for each $\varphi \in L^1_\mu(\varOmega \times M)$. This function, which is $P$-almost surely unique, is called the disintegration of $\mu$ \cite[Chapter 1]{Arnold}.
Let $f:\varOmega \times M \to M$ be a measurable mapping.  
A measure $\mu$ on $\Omega \times M$ is called invariant under $f$ when $\mu$ is invariant under the skew-product mapping $\varTheta (\omega,x) =(\theta \omega,  f(\omega, x))$ and the marginal measure 
of $\mu$ coincides with $P$. 
\footnote{
When $\theta$ is a bimeasurable transformation, it follows from Theorem 1.4.5 in \cite{Arnold} which  
the pushforward measure of $\mu_{\omega}$ by $f(\omega)$ coincides with $\mu _{\theta \omega}$  $P$-almost surely if and only if $\mu$ is invariant under $f$. Such measures $\mu _\omega$ where $\omega \in \varOmega$ are called stationary measures in \cite{BKS}.} 
Given an absolutely continuous invariant probability measure $\mu$ of a measurable mapping $f:\varOmega \times M\to M$, 
 the (operational) forward fiber correlation function $C_{\varphi ,u}(\omega ,n)$ of  $\varphi \in  L^1_{m}( M)$  and $u \in L^\infty_m(M)$
  at $\omega \in \varOmega$ is defined by 
\begin{equation*}
C_{\varphi ,u}(\omega ,n)=\int \varphi  \circ f^{(n)}(\omega)\cdot u dm -\int \varphi d\mu _{\theta ^n\omega}\int u dm,
\end{equation*}
and we call $\ell _\theta^n C_{\varphi ,u} (\omega ,n)$ the (operational) backward fiber correlation function of $\varphi $ and $u $
  at $\omega \in \varOmega$.
(Since $\mu _\omega$ is $P$-almost surely absolutely continuous, $ C _{\varphi ,u}(\cdot ,n)$ is in $L^\infty_P(\varOmega)$ and $\ell _\theta ^nC _{\varphi ,u}(\cdot ,n)$ is well defined.)
The backward fiber correlation functions of $(f,\mu)$ are said to decay exponentially fast in 
a Banach space $E\subset L^\infty_m(M)$ when there exist constants $C>0$ and $0<\tau <1$ (independent of $\omega$) such that 
for any $\varphi \in L^1_m(M)$ and $u \in E$, 
\begin{equation}\label{eq:randomexpdecay}
\left\vert \ell _\theta ^nC _{\varphi ,u} (\omega ,n) \right\vert \leq C\tau ^n\Vert \varphi \Vert _{L^1}\Vert u\Vert _{E}  \quad \text{$P$-a.s.,}
\end{equation} 
where $\Vert \cdot \Vert _E$ is the norm of $E$. 
Similarly, the (operational) integrated correlation functions of $(f,\mu)$ decay exponentially fast in 
a Banach space $E\subset L^\infty_{P\times m}(\varOmega \times M)$ when there exist constants $C>0$ and $0<\tau <1$ (independent of $\omega$) such that 
for any $\varphi \in  L^1_{P\times m}(\varOmega \times M)$  and $u \in E$,  the mapping $\varOmega \ni \omega \mapsto C_{\varphi (\theta ^n\omega),u(\omega)}(\omega ,n)$ is integrable for each $n\geq 1$. Setting $\varphi (\omega )=\varphi(\omega ,\cdot)$, 
\begin{equation}\label{eq:randomexpdecay2}
\left\vert \int C _{\varphi (\theta ^n\cdot) ,u(\cdot)} (\cdot ,n)dP \right\vert \leq C\tau ^n\Vert \varphi \Vert _{L^1}\Vert u\Vert _{E}.
\end{equation} 
The smallest number $\bar \tau$ such that \eqref{eq:randomexpdecay} (or\eqref{eq:randomexpdecay2}) holds for any $\tau >\bar \tau$ is called the rate of exponential decay of backward fiber correlation functions (resp. integrated correlation functions) in $E$.  
When $\theta$ is bimeasurable, 
since $\ell _\theta u=u \circ \theta ^{-1}$ (see Example \ref{examples}), 
then $ \ell ^n_\theta [C _{\varphi (\theta ^n\cdot ) ,u(\cdot)}(\cdot ,n)](\omega)=\ell _\theta ^n[C_{\varphi (\omega),u(\theta ^{-n}\omega)}(\cdot ,n)](\omega) $ $P$-almost surely.
Thus, the exponential decay of backward fiber correlations in $C^{r-1}(M)$ yields 
the exponential decay of forward fiber correlations in $C^{r-1}(M)$ (i.e., \eqref{eq:randomexpdecay} holds, where $\ell _\theta ^nC_{\varphi ,u}(\omega ,n)$ is replaced by $C_{\varphi ,u}(\omega ,n)$)
and also the exponential decay of integrated correlations in $L^\infty(\varOmega,C^{r-1}(M))$. Under these conditions, the mixing of the skew-product mapping is equivalent to the mixing of the base dynamics (see comments in \cite[Subsection 0.2]{Buzzi1}). 
As is well known, any expanding map $f:M\to M$ admits a unique absolutely continuous ergodic invariant probability measure (abbreviated to aceip) on $M$ with a density function of class $C^{r-1}$. In addition, the correlations decay exponentially fast in $C^{r-1}(M)$ (see e.g. \cite{Ruelle}). The aceip of the expanding map $f_0:M\to M$ is denoted by $\mu ^0$. Let $\rho :M\to \C$ be the density function of $\mu ^0$. The rate of exponential decay of correlations of $(f_0,\mu ^0)$ is denoted by $\tau _0$.

Finally, a Banach space $K_P(\varOmega,C^{r-1}(M))$ of random observables as the Kolmogorov quotient (by equality $P$-almost everywhere) of the space is introduced 
\begin{equation}\label{eq:K}
\mathcal K_P(\varOmega,C^{r-1}(M))=\left\{u \in \mathcal L_P^\infty (\varOmega,C^{r-1}(M)) : \omega \mapsto \int u(\omega) dm\; \mathrm{is} \;  \mathrm{constant}\; \text{$P$-} \mathrm{a.s.} \right\}
\end{equation}
endowed with the $L^\infty$ norm. $\mathcal L_P^\infty (\varOmega,C^{r-1}(M))$ is the space of all Bochner measurable mappings $u:\varOmega \to C^{r-1}(M)$ with finite $L^\infty$ norm. 
(In Proposition \ref{BYineq}, it shall be proved that $K_P(\varOmega ,C^{r-1}(M))$ is a Banach space and that $\int u(\cdot)dm$ is measurable.)
As before, a mapping in $\mathcal K_P(\varOmega ,C^{r-1}(M))$ by its equivalence class in $K_P(\varOmega ,C^{r-1}(M))$ is identified.

The following theorem extends Theorems A, B and C in \cite{BKS} to perturbations induced by skew-product mappings whose base dynamics satisfy \eqref{eq:A2}, \eqref{eq:A4} and  \eqref{eq:A3}.
\begin{theorem}
\label{thm:main}
Let $f_0:M\to M$ be an expanding map, 
and $\{ f_{\epsilon} \}_{\epsilon >0}$ be a family of continuous mappings on $(\varOmega ,P)$ with values in $C^r(M,M)$ satisfying \eqref{eq:assumption}. Suppose that $\theta :\varOmega \to \varOmega$ is a measure-preserving transformation satisfying \eqref{eq:A2}, \eqref{eq:A4} and  \eqref{eq:A3}. 
Then, for any sufficiently small $\epsilon >0$, 
there exists a unique absolutely continuous invariant probability measure $\mu ^{\epsilon} $ on $\varOmega \times M$ whose density function $\rho _\epsilon=\frac{d\mu^\epsilon}{d(P\times m)}$ is  in $K_P(\varOmega ,C^{r-1}(M))$ in the notation $\rho _{\epsilon}(\omega)=\rho _\epsilon(\omega ,\cdot) $,  
and we have
\[
\esssup\displaylimits_{\omega\in\varOmega} \left\Vert \rho _\epsilon (\omega)- \rho _0\right\Vert _{ C^{r-1}} \rightarrow 0 \quad \text{as $\epsilon \rightarrow 0$}.
\]

Moreover, for each sufficiently small $\epsilon >0$, the backward fiber correlation functions and the  integrated correlation functions of $(f_{\epsilon}, \mu ^\epsilon)$ decay exponentially fast with rate $0<\tau _\epsilon <1$ in $C^{r-1}(M)$ and in $K_P(\varOmega ,C^{r-1}(M))$, respectively, and we have 
\begin{equation*}
\lim _{\epsilon \rightarrow 0}\tau _{\epsilon}  \leq\max \{ \tau _0 ,\Lambda _r(f_0)\}.
\end{equation*}
\end{theorem}

\begin{remark}
Bogensch\"utz \cite{B1} and Baladi \cite{Baladi97} also investigated stability problems of expanding maps using perturbative spectral approaches.
 Apart from the invertibility of the base dynamics, Theorem \ref{thm:main} differs from Bogensch\"utz's result 
in which he postulated a perturbation lemma for linear cocycles. Therefore, in his result, the "coefficient" $C$ in \eqref{eq:randomexpdecay} may depend on $\omega$, and 
 the integrated correlations 
 may not decay exponentially fast, as demonstrated by Buzzi in \cite[Appendix A]{Buzzi1}. 
Within the setting of mixing base dynamics, Baladi obtained a sharper spectral stability, which yields a more satisfactory result for the decay rate stability; compare \cite[Theorem 5 and Proposition 3.1]{Baladi97} and her Banach space $\mathcal B(\alpha)$ with Theorem \ref{thm:main}, Proposition \ref{Prop6} and $K_P(\varOmega ,C^{r-1}(M))$. 
However, the quasi-compactness of the transfer operator of the skew-product mapping in the Banach space $\mathcal B(\alpha)$ implies the mixing in the skew-product mapping  (in particular, the mixing in the base dynamics). Thus, Baladi's Banach space $\mathcal B(\alpha)$ is not applicable to setting used in this study, in which the base dynamics are not necessarily mixing. 
\end{remark}

\begin{remark}
It follows from Theorem \ref{thm:main} that 
if $(\theta ,P)$ is ergodic, then $(\varTheta _\epsilon,\mu^\epsilon)$ is ergodic for any sufficiently small $\epsilon >0$.
Indeed, let $A \in \mathcal B(\varOmega)\times \mathcal B(M)$ be invariant under $\varTheta _\epsilon$, and suppose that $0<\mu ^\epsilon (A)<1$. 
Then, it follows from Theorem \ref{thm:main} and the invariance of $A$ that for each $B\in \mathcal B(\varOmega)\times \mathcal B(M)$, if the length of $B^\omega =\{ x\in M :(x,\omega)\in B\}$ is $P$-almost surely constant (where the constant is denoted as $\ell(B)$), then
\begin{equation}\label{eq:premixing}
 (P\times m)(A\cap B)=\mu ^\epsilon (A)\cdot (P\times m)(B).
\end{equation}
Let $\Gamma _1=\{ \omega \in \varOmega: m(A^\omega)=0\}$.
Then, noting that $A^\omega =(f(\omega))^{-1}A^{\theta \omega}$ by the invariance of $A$ and that $f(\omega)$ is non-singular with respect to $m$ for each $\omega \in \varOmega$, $\theta ^{-1}\Gamma _1= \Gamma _1$.
Since $(\theta ,P)$ is ergodic and $P(\Gamma _1)\neq 1$ (otherwise, $\mu^\epsilon (A)=0$ by the absolute continuity of  $\mu^\epsilon $), $P(\Gamma _1)=0$.
On the other hand, $\Gamma _2=\{ \omega \in \varOmega :m(A^\omega) =1\} $ is not a full measure set since $\mu ^\epsilon (A)<1$. 
Thus, the set $\Gamma _3=\{ \omega \in \varOmega :0<m(A^\omega) <1\}$ is a positive measure set, and 
we can find a positive measure set $\Gamma \subset  \Gamma _3$ and $B_1,B_2 \in \mathcal B(\varOmega) \times \mathcal B(M)$ such that  $m(B_1^\omega)$ and $m(B_2^\omega)$ are $P$-almost surely constant, $\ell(B_1)=\ell(B_2)\neq 0$, $B^\omega _1\cap A^\omega =\emptyset$ and $B^\omega _2\subset A^\omega$ for each $\omega \in \Gamma$, and $B_1^\omega =B_2^\omega$ for each $\omega \in \varOmega \backslash \Gamma$.
Since these results contradict \eqref{eq:premixing}, $(\varTheta _\epsilon ,\mu^\epsilon)$ is ergodic.
\end{remark}

\begin{example}\label{examples}
We consider examples of measure-preserving transformations satisfying conditions \eqref{eq:A2}, \eqref{eq:A4}, and \eqref{eq:A3}.
The most trivial example is a bimeasurable transformation. When $\theta : \varOmega \to \varOmega$ is  bimeasurable,  $\ell _\theta u(\omega) =u(\theta ^{-1}\omega)$ for each $u\in L^\infty_P(\varOmega)$ and $P$-almost every $\omega \in \varOmega$ since $u(\omega)=u(\theta (\theta ^{-1}\omega))$. 
For each $u\in L^\infty_P(\varOmega,C^{r-1}(M))$,
let us define $\tilde \ell _\theta u:\varOmega\to C^{r-1}(M)$ by $\tilde \ell _\theta u=u \circ \theta ^{-1}$. 
Then, $\tilde \ell _\theta u$ is Bochner measurable since $\tilde \ell _\theta u$ is the composition of the Bochner measurable mapping $u:\varOmega\to C^{r-1}(M)$ and the measurable mapping $\theta ^{-1}:\varOmega\to\varOmega$. It is straightforward to verify that $\tilde \ell _\theta u\in L^\infty_P(\varOmega ,C^{r-1}(M))$ and that $\tilde \ell _\theta $ is a bounded operator on $L^\infty_P(\varOmega ,C^{r-1}(M))$ satisfying  \eqref{eq:A2}, \eqref{eq:A4} and \eqref{eq:A3}.

Now we consider a piecewise smooth mapping $\theta :\varOmega \to \varOmega$ of class $C^1$ on a compact region $\varOmega \subset \R ^d$, i.e., 
$\varOmega$ is the disjoint union of connected and open subsets $\Gamma _1,\ldots ,\Gamma _k$ up to a set of Lebesgue measures $0$ such that $\theta \vert _{\Gamma _j}$ agrees with a $C^1$ map $\theta _j$ defined on a neighborhood of $\overline \Gamma _j$ and $\theta_j$ is a diffeomorphism on the mapped image for each $1\leq j\leq k$. For a detailed study of these mappings, the reader is referred to \cite{GQ}. 
Let $V$ be the normalized Lebesgue measure on $\varOmega$ and define the transfer operator $\ell _{\theta ,V}:L^1_V(\varOmega) \to L^1_V(\varOmega)$ of $\theta $ with respect to $V$ as
\[
\ell _{\theta ,V} u =\sum _{j=1}^k \frac{1_{\Gamma _j} \cdot u}{\vert \det D\theta _j\vert } \circ \theta _j^{-1},\quad u\in L^1_V(\varOmega).
\]  
From the change of variables formula, it follows that $\int \ell _{\theta ,V}u\cdot \varphi dV=\int u\cdot \varphi \circ \theta dV$ for each $u,\varphi \in L^1_V(\varOmega)$ satisfying $u\cdot \varphi \circ \theta \in L^1_V(\varOmega)$ (in particular, $\varphi \in L^\infty_V(\varOmega)$). Thus, if $P$ is an absolutely continuous invariant measure of $\theta$, then the density function $p\in L^1_V(\varOmega)$ of $P$ is a fixed point of $\ell _{\theta ,V} $. It is assumed that $P$ is an absolutely continuous invariant probability measure whose density function $p$ is strictly positive $V$-almost everywhere.  Extensive examples of such measure-preserving transformations $(\theta ,P) are given in $\cite{Baladi}. Then, 
for each $u\in L^\infty_P(\varOmega)$ and $\varphi \in L^1_P(\varOmega)$, we have
\[
\int u \cdot \varphi \circ \theta dP=\int \ell _{\theta ,V} (u\cdot p)\cdot \varphi dV=\int \frac{\ell _{\theta ,V}  (u\cdot p)}{p}\cdot \varphi dP.
\]
Thus, for each $u\in L^\infty_P(\varOmega)$, $\ell _\theta u=\ell _{\theta ,V}  (u\cdot p)/p$ $P$-almost surely. For each $u\in L^\infty_P(\varOmega ,C^{r-1}(M))$, a mapping $\tilde \ell _\theta u:\varOmega \to C^{r-1}(M)$ is defined as $\tilde \ell _\theta u=
[\sum _{j=1}^k  (1_{\Gamma _j} \cdot u\cdot p \cdot \vert \det D\theta _j\vert ^{-1}) \circ \theta _j^{-1}]/p.
$
Since every subspaces of a separable metric space $C^{r-1}(M)$ is itself a separable space (see e.g.~ \cite[Theorem 16.2.b and 16.11]{Willard}), the (weakly) measurable mappings $1_{\Gamma _j}$, $p$,  $\vert \det D\theta \vert ^{-1}$  ($1\leq j\leq k$), and therefore $\tilde \ell _\theta u$, are Bochner measurable by the Pettis measurability theorem. 
Note that 
$
\Vert \tilde \ell _\theta u(\omega)\Vert _{C^{r-1}} \leq \Vert u \Vert _{L^\infty}\left \vert
\ell _\theta 1_{\varOmega} (\omega)\right\vert 
$
$P$-almost surely, since all of $1_{\Gamma _j}$, $p$,  $\vert \det D\theta \vert ^{-1}$ ($1\leq j\leq k$) are independent of $x$.
It follows from this and the fact $\ell _\theta 1_{\varOmega} =1_\varOmega$ (note that $\ell_{\theta ,V} p=p$) that $\tilde \ell _\theta $ is a bounded operator on $L^\infty_P(\varOmega ,C^{r-1}(M))$ satisfying \eqref{eq:A3}. It is straightforward to check by construction that $\tilde \ell _\theta $ satisfies   \eqref{eq:A2} and \eqref{eq:A4}.

Finally, the one-sided shift $\theta : \varOmega\to \varOmega$  is considered: $(\varOmega ,P)=(\tilde \varOmega ^\N, \tilde P^\N)$ is the product space of a probability separable metric space $(\tilde \varOmega,\tilde P)$, in which $(\theta \omega )_j=\omega _{j+1}$ for each $j\in \N=\{0,1,\ldots \}$ and each $\omega =(\omega _0\omega _1\ldots)\in \varOmega $. We note that for each $u\in L^\infty_P(\varOmega)$ and $\varphi \in L^1_P(\varOmega)$, 
\[
\int \left(\int u(\tilde \omega \omega) d\tilde P(\tilde \omega )\right)\cdot \varphi(\omega) dP=
\int  u(\tilde \omega \omega_0\omega _1\ldots )\cdot \varphi(\theta (\tilde \omega \omega _0\omega _1\ldots) )d\tilde P(\tilde \omega)dP(\omega ).
\]
Thus, $\ell _\theta u(\omega)=\int u(\tilde \omega \omega) d\tilde P(\tilde \omega)$ for $P$-almost every $\omega \in \varOmega$. By Fubini's theorem (consider the equivalence between the weak measurability and the Bochner measurability of a mapping $u:\varOmega \to C^{r-1}(M)$), for any $u\in L^\infty_P(\varOmega ,C^{r-1}(M))$,  there exists a Bochner measurable mapping $\tilde \ell _\theta u:\varOmega \to C^{r-1}(M)$ given by
\[
\tilde \ell _\theta u(\omega)=\int u(\tilde \omega \omega) d\tilde P(\tilde \omega),\quad \omega \in \varOmega.
\]
Furthermore, \eqref{eq:A3} for this bounded operator $\tilde \ell _\theta$ on $L^\infty_P(\varOmega ,C^{r-1}(M))$ follows from the Bochner integrability of $\tilde \varOmega \ni \tilde \omega \mapsto u(\tilde \omega \omega)$ for $P$-almost every $\omega \in \varOmega$ (by Fubini's theorem) and the triangle inequality. \eqref{eq:A2} and \eqref{eq:A4} are immediately obtained by construction.

\end{example}

\section{The proof}\label{section:proof}

The proof is started by analyzing the spectrum of "the graph transformation" induced by the transfer operators of the fiber dynamics $f_\epsilon (\omega)$, which is exactly the transfer operator of the skew-product mapping $\varTheta _\epsilon$ with respect to $P\times m$. 
Given a $C^r$ expanding mapping $f:M\to M$, the transfer operator $L(f):C^{r-1}(M) \to C^{r-1}(M)$ is defined as
\[
L(f)u (x)=\sum _{f(y)=x} \frac{u (y)}{\vert \det Df (y)\vert },\quad  x\in M 
\]
for each $u \in C^{r-1}(M).$
As is well known, for each $u \in C^{r-1}(M)$ and $\varphi \in L^1_m(M)$, the change of variables formula yields
\begin{equation}\label{eq:duality}
\int \varphi \cdot L(f)u dm=\int \varphi \circ f \cdot u dm.
\end{equation} 
It is remarked that $C^{r}(M,M)\ni f \mapsto L(f)$ is generally \emph{not} continuous in the norm topology. However, this quantity is continuous in the strong operator topology, as shown below. 
\begin{lemma}\label{lem:measurability}
There exists a $C^r$ neighborhood $\mathcal N(f_0)$ of $f_0$ such that for each $u \in C^{r-1}(M)$, the map $f\mapsto L(f) u$ is a continuous map from $\mathcal N(f_0)$ to 
$C^{r-1}(M)$. 
\end{lemma}
\begin{proof}
To prove this lemma, the argument in \cite[Lemma A.1]{BKS} is adopted. 
Let $\mathcal N(f_0)$ be a small $C^r$ neighborhood of $f_0$ so that any $f\in \mathcal N(f_0)$ is an expanding map.
We recall that all orbits of $f\in \mathcal N(f_0)$ are strongly shadowable: if $\tilde f$ is in a $\epsilon$-neighborhood of $f$ where $\epsilon >0$ is sufficiently small, then for a fixed $x\in  M$, there is a natural bijection between the sets $\{y \vert f(y)=x\}$ and $\{\tilde y \vert \tilde f (\tilde y)=x\}$ such that the distance between paired points is at most $\mathcal O(\epsilon)$. 
Given an integer $0\leq j\leq k$, a straightforward calculation shows that the $j$-th derivative of $L(f)u$ takes the following form:
\begin{equation}\label{eq:strsdw}
(L(f)u)^{(j)} (x)=\sum _{n=1}^{N(j,f)}a_n(u,f;y),
\end{equation}
where $N(j,f)\in \N$ and constant for all $f\in \mathcal N(f_0)$ (as can be seen by shrinking the neighborhood), and the terms of $a_n(u,f;\cdot):M\to \C$ for each $1\leq n\leq N(j,f)$ involve only the $m$-th derivative of $u$, $\vert \det Df \vert ^{-1}$ and $Df^{-1} \circ f$ with $m\leq j$, abusing the notation of the inverse branch of $f$ by $f^{-1}$. Hence, for each  $f\in \mathcal N(f_0)$, $1\leq j\leq k$ (particularly for $j=k$) and $1\leq n\leq N(j,f)$, the $j$-th derivative of $a_n(u,f;\cdot)$ is $\gamma$-H\"older, and the $\gamma$-H\"oder coefficient of the $j$-th derivative of $a_n(u,f;\cdot)-a_n(u,\tilde f;\cdot)$ is bounded by 
$\delta_r(f,\tilde f) \Vert u\Vert _{C^{r-1}}$, 
where $\delta _r(f,\tilde f)$ is a positive number that tends to zero as $f$ converges to $\tilde f$ in the $C^r$-topology. 
The conclusion immediately follows. 
\end{proof}

For simplicity, it is written as
 $L(\epsilon ;\omega)$ and $L_n(\epsilon ;\omega)$
 for $L(f_{\epsilon}(\omega))$ and $L(f_\epsilon^{(n)}(\omega))$, respectively, where $n\geq 1$, $\epsilon >0$ is sufficiently small and $\omega \in \varOmega$. 
For each $u \in L^\infty_P(\varOmega ,C^{r-1}(M))$, a function $\tilde{\mathcal L }_\epsilon u :\varOmega \to C^{r-1}(M)$ is defined as
\[
\tilde{\mathcal L}_\epsilon u(\omega) = L (\epsilon ;\omega)u (\omega), \quad \omega \in \varOmega.
\]
Note that $\tilde{ \mathcal L}_\epsilon u$ is the composition of a mapping $\alpha :\varOmega \times \varOmega \to C^{r-1}(M)$, $\alpha(\omega ,\omega ^{\prime})=L(\epsilon;\omega )u(\omega ^\prime)$ and a measurable mapping $\Lambda :\varOmega \to \varOmega \times \varOmega$, $\Lambda (\omega)=(\omega ,\omega)$. 
From Lemma \ref{lem:measurability} and the continuity of $f_\epsilon$, it follows that if $\epsilon$ is sufficiently small, then for each $\omega ^\prime \in \varOmega$, $\omega \mapsto \alpha(\omega ,\omega ^\prime)$ is a continuous mapping from $\varOmega$ to $C^{r-1}(M)$. 
Furthermore, for each $\omega \in \varOmega$, the mapping $\varOmega \ni \omega ^\prime \mapsto \alpha(\omega ,\omega ^\prime)$ is  measurable since $L(\epsilon;\omega)$ is continuous. Hence, by \cite[Lemma 3.14]{CV77}, $\alpha:\varOmega \times \varOmega \to C^{r-1}(M)$ and $\tilde {\mathcal L}_\epsilon u: \varOmega \to C^{r-1}(M)$ are both measureable. Moreover, reiterating the argument in Example \ref{examples} on Bochner measurability, it is deduced that $\tilde{\mathcal L}_\epsilon u: \varOmega \to C^{r-1}(M)$ is a Bochner measurable mapping.

Now, the weak Lasota-Yorke inequality for expanding maps (see e.g.~ \cite[Lemma 4.2]{BKS}) is adopted: that is, for each $C^r$ expanding map $f:M\to M$, there exists a constant $C_f>0$ such that 
\begin{equation}\label{eq:weakLY}
\Vert L(f)u\Vert _{C^{r-1}} \leq C_f \Vert u\Vert _{C^{r-1}}
\end{equation}
for each $u\in C^{r-1}(M)$. 
Hence, it follows from \eqref{eq:strsdw} and the estimate of $a_n(u,f;\cdot)-a_n(u,\tilde f;\cdot)$ in Lemma \ref{lem:measurability} that if $\epsilon >0$ is sufficiently small, then we have
\[
\Vert \tilde{ \mathcal L} _{\epsilon}u(\omega)\Vert _{C^{r-1}} \leq (C_{f_0}+\delta_r(f_\epsilon (\omega),f_0)) \Vert u(\omega)\Vert _{C^{r-1}}, \quad \text{$P$-a.s.,}
\]
for  each $u\in L^\infty_P(\varOmega ,C^{r-1}(M))$, where the notation $\delta _r(\cdot ,\cdot)$ adopted in the proof of Proposition \ref{lem:measurability} is used, i.e., $\tilde{\mathcal L}_\epsilon$ is a bounded operator on $L^\infty_P(\varOmega ,C^{r-1}(M))$. 
Recalling that $\tilde \ell _\theta $ is a bounded operator on $L^\infty_P(\varOmega ,C^{r-1}(M))$, we can define a bounded operator $\mathcal L_\epsilon :L^\infty_P(\varOmega ,C^{r-1}(M)) \to L^\infty_P(\varOmega ,C^{r-1}(M))$ by
\[
\mathcal L _\epsilon =\tilde \ell _\theta \tilde{\mathcal L}_\epsilon.
\]

$\mathcal L_\epsilon$ is "the transfer operator" of the skew-product mapping $\varTheta _\epsilon$ with respect to $P \times m$. Indeed, if $u\in L^\infty_P(\varOmega,C^{r-1}(M))$, then $\varOmega \times M\ni (\omega ,x)\mapsto u(\omega ,x)$ is measurable using the notation $u(\omega, \cdot)=u(\omega)$ by virtue of \cite[Lemma 3.14]{CV77} together with the fact that $\omega \mapsto u(\omega ,x)$ is measurable for each $x\in M$ (recall the argument in \eqref{eq:A2} above), and that $x\mapsto u(\omega ,x)$ is continuous for each $\omega \in \varOmega$.
 Hence, for each $u\in L^\infty_P(\varOmega ,C^{r-1}(M))$ and $\varphi \in L^1_{P \times m}(\varOmega \times M)$, applying \eqref{eq:A1}, \eqref{eq:A2} and \eqref{eq:duality} together with Fubini's theorem, 
 we have
\begin{align}\label{eq:globalduality}
&\int \varphi \cdot \mathcal L_\epsilon udmdP =\int \left(\int \varphi (\omega,x) \cdot  \ell _\theta[\tilde{\mathcal L} _\epsilon u(\cdot,x)](\omega) dP\right)dm\\ \notag
&= \int \left(\int \varphi (\theta \omega,x) \cdot  \tilde{\mathcal L}_\epsilon u(\omega,x)dP\right)dm
= \int \varphi  (\theta \omega,f_\epsilon (\omega ,x)) \cdot u(\omega,x)dmdP.
\end{align}
Moreover, for each $n\geq 1$ and $u\in L^\infty_P(\varOmega,C^{r-1}(M))$, applying $n$ iterations of \eqref{eq:globalduality} together with \eqref{eq:fibre}, \eqref{eq:A1}, \eqref{eq:A2}, \eqref{eq:duality}, and Fubini's theorem, we have
\begin{equation}\label{eq:F1}
\mathcal L_\epsilon ^nu (\omega)=\tilde \ell _\theta ^n[L_n(\epsilon ;\cdot) u(\cdot)](\omega) \quad \text{$P$-a.s.}
\end{equation}

The following proposition is not difficult to prove but is important.
\begin{proposition}\label{BYineq}
For any $\epsilon > 0$, $\mathcal{L} _{\epsilon}$ preserves a Banach space $K_P(\varOmega ,C^{r-1}(M))$ given in \eqref{eq:K}.
\end{proposition}

\begin{proof}
It is first shown that $K_P(\varOmega ,C^{r-1}(M))$ is a Banach space.  For each $u \in L^\infty_P(\varOmega ,C^{r-1}(M))$,  $I(u;\cdot):\varOmega \to \C$ is defined as $I(u;\omega)=\int u(\omega)dm$. As discussed above \eqref{eq:globalduality},  $\varOmega \times M\ni (\omega ,x)\mapsto u(\omega ,x)$ is measurable, and $I(u;\cdot):\varOmega \to \C$ is measurable by Fubini's theorem.
If $u\in K_P(\varOmega ,C^{r-1}(M))$, then $I(u;\cdot)$ is $P$-almost surely constant. The constant is denoted by $\bar I(u)$. 
 Since the space $L^\infty_P(\varOmega ,C^{r-1}(M))$ is complete, a Cauchy sequence $\{ u_n\}_{n\geq 1} \subset K_P(\varOmega ,C^{r-1}(M))$ has a limit $\bar u$  of $\{u_n\}$ in $L^\infty_P(\varOmega ,C^{r-1}(M))$ with respect to the norm $\Vert \cdot \Vert _{L^\infty}$. 
Hence, it suffices to show that $I(\bar u;\cdot) $ is $P$-almost surely constant. 
We define $\Gamma=\cup _{n\geq 1}\Gamma_n$ with zero measure sets 
$
\Gamma_n=\{ \omega : I(u_n;\omega) \neq \bar I(u_n)\}.
$
Then, it is easily seen that $P(\Gamma)=0$, and $I(u_n;\omega)=\bar I(u_n)$ for all $\omega \in \varOmega \backslash \Gamma$ and $n\geq 1$. We also note that for each $u, v\in L^\infty_P(\varOmega ,C^{r-1}(M))$, $\vert I(u;\omega)-I(v;\omega)\vert \leq \Vert u(\omega) -v(\omega)\Vert _{C^{r-1}(M)}\leq \Vert u-v\Vert _{L^\infty}$ for
$P$-almost every $\omega \in \varOmega$. Thus,
$I(u_n;\cdot)$ $P$-almost surely converges to $I(\bar u;\cdot)$, and $\bar I(u_n)$ converges to a number $\bar I$, and therefore $I(\bar u;\omega)=\bar I$ for $P$-almost every $\omega$ in the full measure set $\varOmega \backslash \Gamma$.

Next, it is shown that $\mathcal L_\epsilon$ preserves $K_P(\varOmega,C^{r-1}(M))$. By \eqref{eq:duality},  
for each $u\in L^\infty_P(\varOmega,C^{r-1}(M))$, 
\begin{align}
\label{eq:preinv}
I(\tilde{\mathcal L} _\epsilon u;\omega) 
=\int L({\epsilon};\omega) u(\omega) dm 
= \int  u(\omega) \cdot 1_M \circ f_{\epsilon}(\omega) dm ,\quad  \omega \in \varOmega,
\end{align}
which coincides with $I(u;\omega)$
since $
1_M\circ f=1_M
$ for any mapping $f :M\rightarrow M$ on $M$. That is, $\tilde{\mathcal L} _\epsilon u \in K_P(\varOmega ,C^{r-1}(M))$ for each $u\in K_P(\varOmega ,C^{r-1}(M))$.

If $u\in L^\infty_P(\varOmega,C^{r-1}(M))$, then $\varOmega \times M\ni (\omega ,x)\mapsto u(\omega ,x)$ is measurable as discussed above. 
Hence, it follows from \eqref{eq:A1}, \eqref{eq:A2} and Fubini's theorem that for each $\varphi \in L^1_P(\varOmega)$
\begin{multline*}
\int \varphi (\omega) \cdot I( \tilde \ell _\theta u;\omega)dP=\int \varphi \cdot \tilde \ell _\theta  u dmdP \\
=\int \left( \int \varphi (\omega ) \cdot \ell _\theta \left[ u(\cdot ,x)  \right] (\omega) dP \right) dm
= \int \varphi (\theta \omega ) \cdot u(\omega ,x)   dP dm, 
\end{multline*}
which, again by Fubini's theorem, coincides with $ \int \varphi (\theta \omega ) \cdot I(u;\omega )   dP$. Specifying $u\in K_P(\varOmega ,C^{r-1}(M))$, it is written as $\int \varphi (\theta \omega) \cdot I(u;\omega)dP=\int \varphi (\omega) \cdot \bar I(u)dP$ since $P$ is an invariant measure. Thus, $\tilde \ell _\theta u$ is also in $K_P(\varOmega ,C^{r-1}(M))$, and 
\begin{equation}\label{eq:d1}
\bar I(\tilde \ell _\theta u)=\bar I(u).
\end{equation}
It immediately follows from this demonstration and \eqref{eq:preinv} that $\mathcal L_\epsilon u=\tilde \ell _\theta \tilde{\mathcal L}_\epsilon u$ is in $K_P(\varOmega ,C^{r-1}(M))$, and the conclusion is obtained.
\end{proof}

The spectrum of $\mathcal L _\epsilon$ is now analyzed by showing that the operator $\mathcal L_\epsilon$ closely matches the operator $L(f_0)$. However, $\mathcal L_\epsilon$ and $L(f_0)$ are not directly relatable because the two operators act on different spaces. 
To obtain a meaningful comparison, the transfer operator of the skew-product mapping $\varTheta _0:(\omega ,x)\mapsto (\theta \omega ,f_0(x))$ is considered.  A bounded operator $\mathcal L_0:L^\infty_P (\varOmega ,C^{r-1}(M)) \to L^\infty_P (\varOmega ,C^{r-1}(M))$ is defined as $\mathcal L_0=\tilde \ell _\theta \tilde{\mathcal L}_0$, where
\[
\tilde{\mathcal L}_0 u (\omega ) =L(f_0)u (\omega), \quad \omega \in \varOmega
\]
for each $u \in L^\infty_P(\varOmega ,C^{r-1}(M))$.
Hereafter, $L(f_0)$ is expressed in the simplified form$L_0$. 
Note that Proposition \ref{BYineq}, \eqref{eq:globalduality}, \eqref{eq:F1} and \eqref{eq:d1} with $\varTheta _\epsilon$, $\tilde{\mathcal L}_\epsilon$ and $\mathcal L_\epsilon$ replaced by $\varTheta _0$, $\tilde{\mathcal L}_0$ and $\mathcal L_0$ hold by the arguments used to develop the proof of Proposition \ref{BYineq} and the respective equations.

The following proposition is essential for proving Theorem \ref{thm:main}.
Let $\sigma (A)$ be the spectrum of a bounded operator $A:E \to E$ on a Banach space $E$. 
In particular, 
the spectrum of  $\mathcal L_0 :K_P(\varOmega ,C^{r-1}(M)) \to K_P(\varOmega ,C^{r-1}(M))$ and $L_0:C^{r-1}(M)\to C^{r-1}(M)$ is denoted as $\sigma (\mathcal L_0)$ and $\sigma (L_0)$, respectively. 

\begin{proposition}\label{Prop6}
$\mathcal{L} _0$ is quasicompact on 
$
K_P(\varOmega,C^{r-1}(M))
$
 with spectral radius $1$, and its spectrum with absolute value $1$ consists only of a simple eigenvalue $1$. 
Moreover, 
\[
\sup \{ \vert z\vert  :  z\in \sigma (\mathcal{L} _0),\; z\neq 1\} \leq \sup \{ \vert z\vert  :  z\in \sigma (L_0),\; z\neq 1\} =:\bar{\tau }_0,
\]
which is strictly smaller than $1$.
\end{proposition}

\begin{proof}
For each $u\in C^{r-1}(M)$, a function $\tilde{\pi} _0 u $ on $M$ is defined as
\begin{align}\label{form}
\tilde{\pi }_0 u= \lim _{n\rightarrow \infty} \frac{1}{n} \sum ^{n}_{k=1}L_0 ^ku . 
\end{align}
From \eqref{eq:weakLY} and the form of $\tilde{\pi }_0$ in (\ref{form}), it follows that $\tilde{\pi }_0$ is a bounded operator on $C^{r-1}(M)$. Thus,
\begin{equation}\label{eq:fixedptofellzero}
L_0 \tilde{ \pi }_0 u=\lim _{n\to \infty} \frac{1}{n} \sum ^n_{k=1} L_0^{k+1}u = \lim _{n\to \infty} \frac{1}{n} \left( \sum _{k=1}^n L^k_0u +L^{n+1}_0u-L_0u\right) = \tilde{ \pi } _0u.
\end{equation}
This equation states that $\tilde{ \pi }_0$ is the projection into the eigenspace of $L_0$ belonging to the  eigenvalue $1$.

Let $\rho _0:=\tilde{\pi} _01_M$. Then $\rho _0\neq 0$. Indeed, by \eqref{eq:duality}, 
\begin{equation}\label{proj}
\int \rho _0dm =\lim _{n\to \infty} \frac{1}{n} \sum _{k=1}^n \int 1_M \cdot 1_M \circ f_0^kdm=1.
\end{equation}
As is well known,
$1$ is the simple eigenvalue of the transfer operator $L(f)$ on $C^{r-1}(M)$ 
 for each $C^r$ expanding map  $f:M\to M$ 
 (see  \cite[Section 2]{BKS}). It therefore 
follows from \eqref{eq:fixedptofellzero} that $\rho _0$ is the unique eigenfunction of $L_0$ up to a constant belonging to the eigenvalue $1$. 

Given $u\in  L^\infty_P(\varOmega,C^{r-1}(M))$, a measurable mapping $\Pi _0 u:\varOmega \to C^{r-1}(M)$ is defined as
\[
\Pi _0u (\omega) =\rho _0 \int u(\omega)dm ,\quad \omega \in \varOmega.
\]
(The measurability of $\Pi _0u$ follows from the proof of Proposition \ref{Prop6}.) 
It follows from $\Vert \Pi _0u\Vert _{L^\infty} \leq \Vert \rho _0 \Vert _{C^{r-1}} \Vert u\Vert _{L^\infty}$ that $\Pi _0$ is a bounded operator on $L^\infty_P(\varOmega,C^{r-1}(M))$.
Moreover, 
$\Pi _0$ is the projection into the eigenspace of $\mathcal{L} _0$ restricted on the Banach space $K_P(\varOmega,C^{r-1}(M))$ belonging to the eigenvalue $1$: \eqref{proj} yields $\Pi _0\Pi _0=\Pi _0$, and it follows from \eqref{eq:fixedptofellzero} that for each $u\in K_P(\varOmega,C^{r-1}(M))$ and $P$-almost every $\omega \in \varOmega$, 
\begin{equation}
\label{eq:b3}
\mathcal L_0\Pi _0u(\omega)=\tilde \ell _\theta [L_0 \rho _0 \cdot I( u,\cdot)](\omega)=\rho _0\bar I(u)=\Pi _0u(\omega),
\end{equation}
where the notations $I(u,\cdot)$ and $\bar I(u)$ adopted in the proof of Proposition \ref{BYineq} are used. Another projector is now defined as
$\Pi _1:=\mathbf{Id} -\Pi _0$, and decompose $\mathcal{L} _0$ into $\mathcal{K} =\mathcal{L} _0\Pi _0$ and $\mathcal{R} =\mathcal{L} _0\Pi _1$. Since 
$\Pi _0K_P(\varOmega,C^{r-1}(M))\cong \mathbb{C} \rho _0$,
$\mathcal{K}$ is a compact operator on $K_P(\varOmega,C^{r-1}(M))$. 
Furthermore, by virtue of \eqref{eq:duality}, \eqref{eq:d1}, and \eqref{eq:b3}, 
\[
\mathcal L_0 \Pi _0 u(\omega) =\rho _0\bar I(\tilde \ell _\theta u)=\rho _0 \int \tilde \ell _\theta u(\omega) dm=\rho _0\int L_0\tilde \ell _\theta u(\omega) dm, \quad \text{$P$-a.s.}
\]
for each $u \in K_P(\varOmega ,C^{r-1}( M))$ is obtained. Thus, by \eqref{eq:A2}, 
$
\mathcal L_0\Pi _j=\Pi _j\mathcal L_0
$ is obtained, where $j=0,1$. In particular, we get for each $n\geq 1$, 
\begin{equation}\label{eq:commute}
\mathcal R ^n =\mathcal L_0^n\Pi _1.
\end{equation}

Similarly, let us define bounded operators $\pi _0, \pi _1$ on $C^{r-1}(M)$ as
\[
\pi _0 u =\rho _0\int  udm  , \quad \pi _1u=u-\pi _0u, \quad u \in C^{r-1}(M).
\]
Then, it is straightforward to check that $\pi _0, \pi _1$ are projections, and that $\pi _0C^{r-1}(M)$ is the one-dimensional eigenspace of $L_0$ belonging to the eigenvalue $1$. In other words, $\pi _0$ coincides with $\tilde{\pi}_0$. Now, $L_0$ is decomposed into 
a compact operator $
K=L_0\pi _0$ and a bounded operator $R=L_0\pi _1
$.
By the approach used to demonstrate \eqref{eq:commute}, it can be observed that $L_0$ preserves $\pi_1C^{r-1}(M)$. We recall that the transfer operator $L(f):C^{r-1}(M)\to C^{r-1}(M)$ of a $C^r$ expanding map $f:M\to M$ is quasi-compact with spectral radius $1$, and its spectrum with absolute value $1$ solely consists of the simple eigenvalue $1$ (see \cite[Section 2]{BKS}). Therefore, $\tau _0<1$, and  
 there exists a constant $C>0$ such that for any $u \in C^{r-1}(M)$ and $n\geq 1$,
\begin{equation*}
\Vert L_0^n \pi _1u \Vert _{C^{r-1}(M)}  \leq C \bar{\tau }_0^n\Vert u \Vert _{C^{r-1}(M)}.
\end{equation*}
It follows from \eqref{eq:A3}, \eqref{eq:F1}, and \eqref{eq:commute} 
that for any $u \in K_P(\varOmega,C^{r-1}(M))$, $n\geq 1$, and $P$-almost every $\omega \in \varOmega$,
\begin{align*}
\Vert \mathcal{R} ^nu (\omega)\Vert _{C^{r-1}} =\Vert \tilde \ell _\theta ^n( L_0 ^n[\Pi _1u (\cdot)])(\omega) \Vert  _{C^{r-1}}
& \leq  \Vert L^n_0 \pi_1[u(\omega)]\Vert _{C^{r-1}} \\
 &\leq C\bar{\tau } ^n _0 \Vert u (\omega)\Vert _{C^{r-1}},
\end{align*}
i.e., the spectral radius of $\mathcal R$ is bounded by $\bar \tau_0$.
The conclusion follows from a straightforward check that the spectral radius of $\mathcal{L} _0$ is $1$. 
\end{proof}

Now, the Baladi-Young perturbation lemmas can be applied 
to families of linear operators. The relevant lemmas are  
 Lemmas $1$, $2$, $3$, and the comment below Lemma $1$ in \cite{BY}. Let 
 $X=K_P(\varOmega,C^{r-1}(M))$, $T_0=\mathcal{L} _0$, $T_{\epsilon} =\mathcal{L} _{\epsilon}$, $X_0=\Pi _0 K_P(\varOmega,C^{r-1}(M))$, $X_1=\Pi _1 K_P(\varOmega,C^{r-1}(M))$, $\kappa _0=1$, $\kappa _1=\max \{ \bar \tau _0, \Lambda_r(f_0)\}$. $\kappa $ is arbitrarily close to (and slightly bigger than) $\kappa _1$, 
and $\Pi _0$ and $\Pi _1$ are the projections given in the proof of Proposition \ref{Prop6}. 
Indeed, it is straightforward to verify that hypotheses (A.1) and (A.3) in the lemmas are satisfied by Proposition \ref{BYineq} and \ref{Prop6}, and that hypothesis (A.2) follows from \cite[Lemma A.1]{BKS}. 
From the Baladi-Young perturbation lemmas, it follows that there exists a family of decompositions $K_P(\varOmega,C^{r-1}(M))=X_0^{\epsilon} \oplus X_1^{\epsilon}$, $\epsilon >0$, in which the projections $\Pi _0 ^{\epsilon}:X^{\epsilon}_0\oplus X^{\epsilon}_1\rightarrow X^{\epsilon}_0$ satisfy 
\begin{equation}\label{stc.stb}
\Vert \Pi _0 -\Pi _0 ^{\epsilon} \Vert _{L^\infty} \rightarrow 0 \quad \mathrm{as} \; \epsilon \rightarrow 0,
\end{equation}
and
\begin{equation*}
\sigma (\mathcal{L}_{\epsilon} \vert _{X_0^{\epsilon}}) \rightarrow \sigma (\mathcal{L}_0\vert _{X_0})\quad \mathrm{as} \; \epsilon \rightarrow 0,
\end{equation*}
in terms of the Hausdorff distance and using the notation $\bar{ \tau}_{\epsilon} = \sup \{ \vert z\vert : z \in \sigma (\mathcal{L} _{\epsilon} \vert _{X^{\epsilon}_1}) \}$, 
we have
\begin{equation}\label{UpperBound2}
\lim _{\epsilon \rightarrow \infty} \bar{\tau }_{\epsilon} \leq \kappa _1.
\end{equation}

Let $\bar{\lambda  }_{\epsilon} \in \sigma (\mathcal{L} _{\epsilon}\vert _{X^{\epsilon} _0})$ be the simple eigenvalue that converges to $1$, and let $\rho _{\epsilon}:=\Pi _0^{\epsilon}1_{\varOmega \times M}$. 
It will now be shown that $\bar{\lambda }_{\epsilon}=1$  for any sufficiently small $\epsilon >0$. 
For $P$-almost every $\omega \in \varOmega$, 
\begin{align*}
\int  \rho _{\epsilon} (\omega)dm
=   \int  \rho _0dm- \int  (\rho _0-\rho _\epsilon (\omega ))dm 
\geq \int  \rho _0dm-  \Vert  \rho _{\epsilon}-\rho _0 \Vert _{L^\infty} is obtained.
\end{align*}
From \eqref{proj} and \eqref{stc.stb}, it follows that $ \int  \rho _{\epsilon} (\omega) dm>0$ for any sufficiently small $\epsilon >0$ and $P$-almost every  $\omega \in \varOmega$. 

On the other hand, $\bar \lambda _\epsilon \rho _\epsilon (\omega ,x)=\tilde \ell _\theta \tilde{\mathcal L}_\epsilon \rho _\epsilon (\omega ,x)$ for each $x\in M$ and $P$-almost every $\omega \in \varOmega $. Therefore, by \eqref{eq:duality} and \eqref{eq:d1}, $P$-almost surely we have 
\begin{align*}
\int  \rho _{\epsilon}(\omega)dm = \bar{\lambda  }_{\epsilon}  ^{-1}\bar  I(\tilde \ell _\theta \tilde{\mathcal L}_\epsilon\rho _\epsilon )=\bar{\lambda  }_{\epsilon}  ^{-1}\bar I( \tilde{\mathcal L}_\epsilon\rho _\epsilon) 
=\bar{\lambda  }_{\epsilon}  ^{-1}\int  \rho _{\epsilon}(\omega) \cdot 1_M\circ f_\epsilon (\omega)dm,
\end{align*}
which coincides with $\bar \lambda _\epsilon ^{-1}\int \rho_\epsilon(\omega)dm $.
This implies that $\bar{\lambda }_{\epsilon}=1$ for any sufficiently small $\epsilon >0$.

A measure $\mu ^\epsilon$ on $\varOmega \times M$ is defined as $\mu ^\epsilon (d\omega,dx) =\rho_\epsilon(\omega ,x) (P\times m)(d\omega ,dx)$.
By virtue of \eqref{eq:globalduality} and noting that $\bar{\lambda }_{\epsilon}=1$, $\mu ^\epsilon $ is invariant with respect to $\varTheta _\epsilon$.
Furthermore, it follows from Proposition \ref{Prop6} that $\mathcal L_\epsilon$ is quasi-compact on $K_P(\varOmega,C^{r-1}(M))$ with spectral radius $1$, and that its spectrum with absolute value $1$ solely consists of the simple eigenvalue $1$ for each small $\epsilon >0$.
This implies that 
when the essential spectral radius  of $\mathcal L_\epsilon$ is denoted by $\hat \kappa _\epsilon$, the following inequality holds for any $n\geq 1$ and $u\in K_P(\varOmega ,C^{r-1}(M))$
\[
\Vert \mathcal L_\epsilon ^n u\Vert _{L^\infty} \leq \Vert \Pi _0^\epsilon u\Vert _{L^\infty} +\mathcal O(\hat \kappa _\epsilon ^n) \Vert \Pi _1^\epsilon u\Vert _{L^\infty}.
\]
This inequality is bounded by $C\Vert u\Vert _{L^\infty}$, where the constant $C>0$ is independent of $u$ and $n$. 
Hence, we can define a bounded operator $\tilde{\Pi }^{\epsilon}_0 $ on $K_P(\varOmega ,C^{r-1}(M))$ of the form
\begin{align*}
\tilde{ \Pi } ^{\epsilon}_0 u=\lim _{n\rightarrow \infty}\frac{1}{n} \sum ^{n}_{k=1}\mathcal{L} _{\epsilon}^k u ,\quad u \in K_P(\varOmega,C^{r-1}(M)).
\end{align*}
As in the proof of Proposition \ref{Prop6}, it can be verified that $\tilde{\Pi }^\epsilon _0$ coincides with the eigenprojection $\Pi ^\epsilon _0:K_P(\varOmega,C^{r-1}(M))\rightarrow X^{\epsilon} _0$. Thus, $\mu ^\epsilon$ is a probability measure on $\varOmega \times M$ (in particular, the disintegration $\mu ^\epsilon _\omega$ of $\mu^\epsilon$ is $P$-almost surely a probability measure on $M$): by \eqref{eq:globalduality},
\[
\mu ^\epsilon (\varOmega \times M) =\lim _{n\to \infty} \sum ^n_{k=1}\int \mathcal L_\epsilon ^k 1_{\varOmega \times M}\cdot 1_{\varOmega \times M}dmdP=1.
\] 
Hence, recalling that $\pi_{\varOmega}^{-1}\Gamma =\Gamma \times M$ for each $\Gamma \in \mathcal F$, we have
\[
\mu^\epsilon(\pi _{\varOmega}^{-1}\Gamma)=\int _{\Gamma }\mu _{\omega}^\epsilon(M)dP=P(\Gamma),
\]
which demonstrates that $\mu^\epsilon$ is a unique absolutely continuous invariant probability measure with the density function $\rho_\epsilon$ in $K_P(\varOmega,C^{r-1}(M))$. 
Furthermore, from \eqref{stc.stb},
$\rho _{\epsilon}  $  converges to the density function $\rho _0$ of the absolutely continuous ergodic invariant probability measure $\mu ^0$ of $f_0$ with respect to the norm $\Vert \cdot \Vert _{L^\infty}$. 

Since the eigenprojection $\Pi ^{\epsilon}_0 :K_P(P,E)\rightarrow X^{\epsilon} _0\cong \mathbb{C} \rho _{\epsilon}$ is unique, it can be easily confirmed that $\Pi ^\epsilon_0$ coincides with a bounded operator $\hat{\Pi }^\epsilon _0$ on $K_P(\varOmega ,C^{r-1}(M))$ given by
\begin{align*}
\hat{\Pi} ^{\epsilon}_0 u(\omega)=\rho_{\epsilon}(\omega)\int u(\omega)dm,\quad  \omega \in  \varOmega
\end{align*}
(which $P$-almost surely coincides with $\rho _{\epsilon}\bar I( u)$) for each $u \in K_P(\varOmega,C^{r-1}(M))$. From the argument used to prove Proposition \ref{Prop6}, it can also be verified 
that $\mathcal L_\epsilon$ preserves $\Pi ^\epsilon _1K_P(\varOmega ,C^{r-1}(M))=(\mathbf{Id} -\Pi ^\epsilon _0)K_P(\varOmega ,C^{r-1}(M))$.
On the other hand, by  \eqref{eq:A2}, \eqref{eq:A4}, \eqref{eq:duality}, and \eqref{eq:F1}, 
 for each $\varphi \in L^1_m(M)$ and $u\in C^{r-1}(M)$, we have a standard rewriting of the backward correlations: 
\begin{align*}
\ell _\theta^n C_{\varphi ,u}(\omega, n) &=\int \varphi \cdot \tilde \ell _\theta ^n[L_n(\epsilon;\cdot) u](\omega)dm-\int \varphi  \cdot \tilde \ell _\theta ^n[\mathcal L_\epsilon ^n\rho _\epsilon(\theta ^n \cdot)](\omega)  dm  \int udm  \\
&=\int \varphi \cdot \mathcal L^n_\epsilon u(\omega)dm-\int \varphi  \cdot \mathcal L_\epsilon ^n\rho _\epsilon(\omega)  \cdot \left( \int udm\right)  dm ,
\end{align*}
which concides with $\int \varphi \cdot \mathcal L_\epsilon^n \Pi _1^\epsilon u(\omega)dm$ for $P$-almost every $\omega \in \varOmega$ by $\hat \Pi^\epsilon _0=\Pi^\epsilon _0$. (In the second equality, we use relation $\tilde \ell _\theta [u(\theta \cdot)] =u(\cdot)$, which holds for each $u\in L^\infty_P(\varOmega ,C^{r-1}(M))$ because $\int \varphi \cdot \tilde \ell _\theta [u\circ \theta ] dP=\int (\varphi \cdot u)\circ \theta dP=\int \varphi \cdot udP$ for any $\varphi \in L^1_P(\varOmega)$.)
Thus, $\ell _\theta^n C_{\varphi ,u}(\omega, n) $ is bounded by $C\bar \tau _\epsilon ^n  \Vert \varphi \Vert _{L^1}\Vert u\Vert _{C^{r-1}}$ for $P$-almost every $\omega \in\varOmega$, where $C>0$ is a constant independent of $\omega$ and $n$.
Similarly, for each $\varphi \in L^1_{P \times m}(\varOmega \times M)$ and $u \in K_P(\varOmega,C^{r-1}(M))$, 
\begin{multline*}
\left\vert \int  C_{\varphi(\theta ^n\cdot) ,u(\cdot)}(\cdot, n)dP\right\vert = \left\vert \int \varphi \cdot \mathcal L_\epsilon^n udmdP-\int \varphi  \cdot \mathcal L_\epsilon ^n\rho _\epsilon \cdot \bar I(u) dmdP \right\vert \\
 = \left\vert \int \varphi \cdot \mathcal L_\epsilon^n \Pi _1^\epsilon udmdP \right\vert
\leq C\bar \tau _\epsilon ^n  \Vert \varphi \Vert _{L^1}\Vert u\Vert _{C^{r-1}}
\end{multline*}
where the constant $C>0$ is independent of $\omega$ and $n$. 
Moreover, 
it is straightforward to see that $\bar \tau _0$ and 
 $\bar \tau _\epsilon$ equal the rate $\tau _0$ of the exponential decay of correlations of $(f_0,\mu^0)$ and the rate $\tau _\epsilon$ of exponential decay of integrated/backward fiber correlations of $(f_\epsilon ,\mu^\epsilon )$, respectively (see e.g.~ \cite[Remark 2.3]{Baladi}). Finally, $\lim _{\epsilon \to 0}\bar \tau _\epsilon \leq \kappa _1=\max \{  \tau _0 ,\Lambda _r(f_0)\}$ by \eqref{UpperBound2}, and we complete the proof of Theorem \ref{thm:main}.

\section*{Acknowledgments}
The author would like to express his deep gratitude to his supervisor, Professor Sigehiro Ushiki for his elaborated guidance and considerable encouragement.

\bibliographystyle{my-amsplain-nodash-abrv-lastnamefirst-nodot}
\bibliography{MATHabrv,APBKS.ref}

\end{document}